\documentclass[a4paper, 10pt, conference]{ieeeconf}      
\usepackage{graphicx,times,amsmath,color,amssymb}
\newtheorem{theorem}{Theorem}[section]
\newtheorem{lemma}{Lemma}[section]
\newtheorem{constraint}{Constraint}[section]
\newtheorem{remark}{Remark}[section]
\newtheorem{assumption}{Assumption}[section]

\newtheorem{problem}{Problem}[section]

\renewcommand{\theequation}{\thesection.\arabic{equation}}
\def\QED{~\rule[-1pt] {8pt}{8pt}\par\medskip}

\renewcommand{\theequation}{\thesection.\arabic{equation}}

\def\QED{~\rule[-1pt]{5pt}{5pt}\par\medskip}

\def\aa{{\alpha}}
\def\ee{{\epsilon}}

\def\be{\begin{equation}}

\def\bi{\begin{itemize}}
\def\ee{\end{equation}}
\def\ei{\end{itemize}}

\def\eps{\epsilon}
\def\xh{\hat{x}}

\def\xt{\tilde{x}}

\def\z1{z^{-1}}
\def\la{\label}

\def\bb{\beta}

\def\mR{\mathbb{R}}

\def\xh{\hat{x}}

\IEEEoverridecommandlockouts                              
\title{\LARGE \bf
Circumnavigation Using Distance Measurements (Extended Version)
}

\author{Iman Shames, Soura Dasgupta, Bar\i\c{s} Fidan, and Brian. D. O. Anderson
\thanks{This work is supported by NICTA, which is funded by the Australian Government as represented by the Department of Broadband, Communications and the Digital Economy and the Australian Research Council through the ICT Centre of Excellence program, and by US NSF grants  ECS-0622017, CCF-072902, and CCF-0830747.}
\thanks{Iman Shames, Brian. D. O. Anderson, and Bar\i\c{s} Fidan are with Australian National University and National ICT Australia. Soura Dasgupta is with Department of Electrical and Computer Engineering, University of Iowa.
        {\tt\small \{Iman.Shames, Brian.Anderson, Baris.Fidan\}@anu.edu.au, Dasupta@engineering.uiowa.edu}, }
}

\begin{document}

\maketitle
\thispagestyle{empty}
\pagestyle{empty}

\begin{abstract}
Consider a stationary agent A at an unknown location and a mobile agent
 B that must move to the vicinity of and then circumnavigate A at a prescribed distance from A. In doing
 so, B can only measure its distance from A, and knows its own position
 in some reference frame. This paper considers this problem, which has
 applications to surveillance or maintaining an orbit. In many of these
 applications it is difficult for B to directly sense the location of A,
 e.g. when all that B can sense is the intensity of a signal emitted by
 A. This intensity does, however provide a measure of the distance. We
 propose a nonlinear periodic continuous time control law that achieves
 the objective. Fundamentally, B must exploit its motion to estimate the
 location of A, and use its best instantaneous estimate of where A
 resides, to move itself to achieve the ultimate circumnavigation objective. The
 control law we propose marries these dual goals and is globally
 exponentially convergent. We  show through simulations that it also permits B to
 approximately achieve this objective when A experiences slow, persistent
 and potentially nontrivial drift.
\end{abstract}

\section{Introduction}
In surveillance missions the main objective is to obtain information about the target of interest by monitoring it for a period of time. In most cases it is desirable to monitor the target by circumnavigating it from a prescribed distance. In recent years this problem has been addressed in the context of autonomous agents, where an agent or a group of agents accomplish the surveillance task. This problem has been extensively studied for the case where the position of the target is known and the agent(s) can measure specific information about the source,   such as distance, power, angle of arrival, time difference of arrival, etc. See \cite{ShamesFidanAnderson_CDC_08,SinhaGhose_ACC_2005,KimSugie_Auto_2007} and references therein. However, in many situations, knowing the position of the target is not practical, e.g.~if one wants to find and monitor an unknown source of an electromagnetic signal. This paper addresses the problem where the position of the source is unknown, only one agent is involved, and the only information continuously available to the agent is its own position and its distance (not relative position) from the source.

There are other recent papers that consider a problem related to the one addressed in this paper. For instance \cite{CaoMorse_ECC_07} and \cite{CaoMorse_ACC_07} using concepts from switched adaptive control, consider the case where an agent must move itself to a point at a pre-set distance from three sources with unknown position in the plane using distance measurements. Another related paper is \cite{DandachFidanDasguptaAnderson_SCL_08}, where the agent's objective is to estimate the position of the source using distance measurements only.

The problem addressed in this paper can be studied as a \emph{dual control} problem as well. In a dual control problem the aim is identify the unknown parameters of the system and achieve a control objective simultaneously \cite{Feldbaum60_Dual}. In this case the mobile agent must estimate the location of the target from the distance measurements, and use this estimate  to  execute its control law for achieving its objective of circumnavigating the target from a prescribed distance. 

To estimate the target  location the agent must move in a trajectory that is not confined to a straight line in two dimensions and to a plane in three dimensions.  To perform robust estimation this avoidance of collinear/coplanar motion must be {\it persistent}  in a sense described in \cite{DandachFidanDasguptaAnderson_SCL_08}.
A key feature of the circumnavigation objective is that it ensures this requirement thus aiding the estimation task.

The structure of the paper is as follows. In the next section the problem is formally defined, and in Section \ref{sec:algo} the proposed algorithm is introduced. The analysis of the control laws is presented in Section \ref{sec:ana}. In Section \ref{pe} the persistent excitation condition on the signals is established, and The proof of exponential stability of the system is provided. In Section \ref{sA} a method to choose one of the parameters in the control laws is presented. The simulation results are presented in Section \ref{sec:sim}. In the final section future directions and concluding remarks are presented.

\section{Problem Statement}
In what follows we formally define the problem addressed in this paper and introduce relevant assumptions.
\begin{problem}
Consider a source at an unknown constant position $x$ and an agent at $y(t)$ in $\mathbb{R}^n$ ($n\in\{2,3\}$)  at time $t\in[0,\infty]$. Knowing $y(t)$, a desired distance $d$, and the measurement 
\begin{equation}\label{D}
D(t)=\|y(t)-x\|
\end{equation}
find a control law that ensures that asymptotically, $y(t)$ moves on a trajectory at a distance $d$ from $x$.
\end{problem}
Here as in the rest of the paper $\|.\|$ denotes 2-norm. For convenience in the rest of the paper, we impose the following constraint:
\begin{constraint}
The agent trajectory $y(t):\mathbb{R}\mapsto\mathbb{R}^n$, is to be  twice differentiable. Further, there exists $M_0>0$, such that $\forall t\in\mathbb{R}: \quad \|y(t)\| + \|\dot{y}(t)\|+\|\ddot{y}(t)\| \leq M_0$.
\end{constraint}
This constraint ensures that the motion of the agent can be executed with finite force. One can break down the problem into the following two sub-problems:
\begin{enumerate}
\item How one can estimate $x$?
\item How one can make the agent move on a trajectory at a distance $d$ from $x$, which, in a sense to be made precise in the sequel,  persistently spans $\mathbb{R}^n$?
\end{enumerate}
The first sub-problem is addressed in \cite{DandachFidanDasguptaAnderson_SCL_08}. However, for the algorithm of \cite{DandachFidanDasguptaAnderson_SCL_08} to work one requires that $\forall t\in\mathbb{R}$,
\begin{equation*}
\alpha_1 I \leq \int_t^{t+T_1} \dot{y}(\tau)\dot{y}(\tau)^\top d\tau \leq \alpha_2 I
\end{equation*}
for some $\alpha_1$,$\alpha_2$, and $T_1$ which are strictly positive. This condition is the well-known persistent excitation (p.e.) condition. This condition requires that the agent in $n$-dimensional space must persistently avoid a trajectory that is confined to the vicinity of a single $(n-1)$-dimensional hyperplane.
As noted earlier, the circumnavigation objective in fact assists in maintaining p.e.
\setcounter{equation}{0}
\section{Proposed Algorithm}\label{sec:algo}
We first present an algorithm for estimating $x$. To this end for $\alpha >0$ generate
\be \la{z1}
\eta(t)=\dot{z}_{1}(t) = -\alpha z_{1}(t)+\frac{1}{2}D^{2}(t), 
\ee
\be \la{z2}
m(t) =\dot{z}_{2}(t) = -\alpha z_{2}(t)+\frac{1}{2}y^\top(t)y(t),
\ee
\be \la{z3}
V(t) =\dot{z}_{3}(t) = -\alpha z_{3}(t)+y(t), 
\ee
where $z_1(0)$, and $z_2(0)=0$ are arbitrary scalars, and $z_3(0)$ is an arbitrary vector. Note that the generation of $\eta(t)$, $m(t)$ and $V(t)$ requires simply the measurements $D(t)$ and the knowledge
of the localizing agent's own position, and can be performed without explicit differentiation.

Define now the estimator:
\be \la{aloc}
\dot{\hat{x}}(t) = -\gamma V(t)(\eta(t)-m(t)+V^\top(t)\hat{x}(t)),
\ee
where $\hat{x}(t)$ denotes the estimate of $x$ at time $t$, and $\gamma>0$ is the adaptive gain. This estimator is in fact identical to the one presented in \cite{DandachFidanDasguptaAnderson_SCL_08}. Later, we shall show that under suitable conditions $\xh$ approaches $x$.

Define 
\be \la{dh}
\hat{D}(t) =\|y(t)-\hat{x}(t)\|,
\ee
 and the control law
\be \la{claw}
\dot{y}(t)= \dot{\hat{x}}(t) -\left [ (\hat{D}^2(t)-d^2) I- A(t)\right ](y(t)-\hat{x}(t)),
\ee
where $A(\cdot):\mR\rightarrow \mR^{n\times n}$ obeys four conditions captured in the assumption below. As will be proved in the next section, this control law aims at moving the agent so that $\hat{D}$ converges to $d$, i.e. the agent takes up the correct distance from the estimated of the source position, $\hat{x}$. If also $\hat{x}$ converges to $x$, then  $D$ converges to $\hat{D}$, hence $D$ converges to $d$. 
\begin{assumption} \la{aa}
\begin{itemize}
\item[(i)] There exists a $T>0$ such that
for all $t$,
\be \la{tper}
A(t+T)=A(t).
\ee
\item[(ii)] $A(t)$ is skew symmetric for all $t$.
\item[(iii)] $A(t)$ is differentiable everywhere.
\item[(iv)] the derivative of the solution of the differential equation below is 
persistently spanning.
\be \la{y*}
\dot{y}^*(t)= A(t)y^*(t).
\ee
for any arbitrary nonzero value of $y^*(0)$. More precisely, there exists a $T_1>0$, and $\aa_i>0$ such that for all $t\geq 0$ there holds
\be \la{ps}
\aa_1  \|y^*(t)\|^2 I\leq \int_{t}^{t+T_1} \dot{y}^*(\tau)\dot{y}^*(\tau)^\top d \tau \leq \aa_2 \|y^*(t)\|^2 I.
\ee
\end{itemize}
\end{assumption}
A consequence of the fact that $A(t)$ is skew symmetric is that for all $\nu\in \mR^n$ and $t\geq 0$
\be \la{skew}
\nu^\top A(t)\nu=0.
\ee
We note that the results below hold even if $A(t)$ were permitted to lose  differentiability
 at a countable number of points. However, that would imply that $\dot{y}$ would lose differentiability
at these same points, resulting in the physically unappealing need for an impulsive force to act on $y(t)$.

As will be shown in Section \ref{sA}, in two dimensions, with $E$ the rotation matrix,
\be \la{E}
E=\left[ \begin{array}{cc} 0 & 1\\-1 &0\end{array}\right ],
\ee
and $a$ any real scalar, it suffices to chose $A(t)$ as the constant matrix $aE$. In three dimensions however,
the selection of $A(t)$ is more complicated. Specifically for (\ref{ps}) to hold with a constant $A$, $A$ must be nonsingular. No $3\times3$
skew symmetric matrix is however nonsingular, leading to the need for a periodic $A(t)$, whose selection is described in Section 
\ref{sA}. To summarize, the overall system is
described by  (\ref{D}) - (\ref{claw}), with $z_i$, $\xh$, and $y$, serving 
as the underlying state variables.
\setcounter{equation}{0}
\section{Analysis}\label{sec:ana}
We shall begin by establishing that the quantities $\eta(t)$ and $m(t)$, computable from the measurements as described in the previous section, can be used to estimate $V(t)^\top x$. Of course, our ultimate goal is to estimate $x$. This will be achieved, but is harder. 

Observe from  (\ref{z1}) that  $\eta(t)=\dot{z}_1(t).$
Thus, one obtains  (recalling that $x$ is constant):
\be \la{n1}
\dot{\eta}(t)=-\aa \eta (t) + \dot{y}(t)^\top(y(t)-x).
\ee
Similarly,
\be \la{n2}
\dot{m}(t)=-\aa m (t) + \dot{y}(t)^\top y(t);
\ee
\be \la{n3}
\dot{V}(t)=-\aa V(t) + \dot{y}(t).
\ee
\be \la{vr}
\frac{d}{dt} \left [ \eta (t)- m(t) +V^\top(t)x\right ]= -\aa \left [ \eta (t)- m(t) +V^\top(t)x\right ],
\ee
which implies the following Lemma.
\begin{lemma}\label{lem:Lemma1}
For all $t_0$ and $t\geq t_0$, there holds:
\be \la{exp}
\eta (t)- m(t) +V^\top(t)x=\left [ \eta (t_0)- m(t_0) +V^\top(t_0)x\right ]e^{-\aa(t-t_0)}.
\ee
\end{lemma}
As foreshadowed in the last section, we now present a lemma that   shows that the agent located at $y(t)$
 moves to a trajectory maintaining a constant distance $d$  from the {\it estimated} position of the agent at position $x$. 
\begin{lemma} \la{L1} 
Consider (\ref{claw}) under Assumption \ref{aa}. Suppose there exists a $\delta >0$ such that in (\ref{dh})
$\hat{D}^2(0)>\delta$. Then $ \hat{D}^2(t) $ converges exponentially to $d^2$, and there holds for all $t\geq 0$
\be \la{dbd}
\hat{D}^2(t) > \min  \{ \delta, d^2 \}
\ee
\end{lemma}
\begin{proof}
Because of (\ref{skew}) one obtains that,
\begin{eqnarray}
\frac{d}{dt}\left \{ \hat{D}^2(t) -d^2 \right \} &=&2 (\dot{y}(t)-\dot{\hat{x}}(t) )^\top 
(y(t)-\hat{x}(t) )\nonumber \\
&=&-2( \hat{D}^2(t) -d^2)\hat{D}^2(t). \la{dd}
\end{eqnarray}
Observe that $\hat{D}$ is bounded and continuous. 
Consider first the case where  $\delta >d^2$.  Then the derivative above is initially negative, i.e. $\hat{D}^2(t) $ declines in value. By its continuity 
for $\hat{D}^2(t)$ to become less $d^2$,  at some point it must equal $d^2$, when  $\hat{D}(t)$ will stop changing.  Since throughout this time $\hat{D}^2(t) \geq d^2$, 
convergence of  $\hat{D}^2(t) $ to $d^2$ occurs at an exponential
rate and $\hat{D}^2(t) \geq d^2$ for all $t$.
On the other hand if $\delta \leq d^2$, then $\hat{D}^2(t)>\delta$ for all $t$, as the derivative of $\hat{D}^2(t)$ is nonnegative. Again exponential convergence 
of  $\hat{D}^2(t) $ to $d^2$ 
occurs. 
\end{proof}
Now define $\tilde{x}(t) = \hat{x}(t)-x$. We have the following Lemma, which is the first of two aimed at establishing a Lyapunov function with certain properties for the overall system. 
\begin{lemma} \la{lbd}
Consider the system defined in (\ref{D}) -  (\ref{claw}),  subject to
the requirement that $\hat{D}(0)>0$ and Assumption \ref{aa}. Define:
\begin{eqnarray} \la{L}
\begin{split}
L(t)&= \frac{1}{4\aa} \left ( \eta(t) - m(t) +V^\top(t) x \right )^2 + \frac{1}{2\gamma} \xt^\top(t)\xt (t)\\
& +\frac{1}{4} \left ( \hat{D}^2(t) -d^2 \right )^2+\sum_{i=1}^3 L_i(t),
\end{split}
\end{eqnarray}
where $L_1(t)= \left ( z_1(0)e^{-\aa  t}\right )^2$, $L_2(t)= \left ( z_2(0)e^{-\aa t}\right )^2$, and $L_3(t)= \left \|\left ( z_3(0)e^{-\aa t}\right )^2 \right \|^2$.
Then, whenever $L(t)$ is bounded so also are the 
state variables
$\xh(t)$, $y(t)$, $z_i(t)$, $\eta(t)$, $m(t)$, and $V(t)$.
\end{lemma}
\begin{proof}
The second and third  terms on the right hand side (\ref{L}) ensure the boundedness of $\xh(t)$, $y(t)$ and $D(t)$. Thus as $L_i(t)$ are bounded (\ref{z1}), (\ref{z2}) and (\ref{z3}) ensure that the $z_i(t)$ are bounded as well.
\end{proof}
We now define, for $\Delta>0$ the set
\be \la{lset}
S(\Delta) = \{ [\xh^\top , y^\top, z_1, z_2, z_3^\top]^\top | L \leq \Delta    \}.
\ee
Because of Lemma \ref{lbd}, the set $S(\Delta)$ is compact.

Next we identify an invariant set for (\ref{D}) - (\ref{aloc}).
\begin{lemma} \la{linv}
Consider the system defined in (\ref{D}) - (\ref{claw}). Define the set ${\mathcal S_I}$ as the set of vectors $[\xh^\top , y^\top, z_1, z_2, z_3^\top]^\top$ satisfying, $\xh=x$, $\|y-x\| =d$, and $z_1-z_2+z_3^\top x =\frac{x^\top x}{2\aa}$. Then $[\xh^\top(0) , y^\top(0), z_1(0), z_2(0), z_3^\top(0)]^\top \in {\mathcal S_I}$ 
 implies $[\xh^\top(t) , y^\top(t), z_1(t), z_2(t), z_3^\top(t)]^\top \in {\mathcal S_I}$  for all $t\geq 0$.
\end{lemma}
\begin{proof}
First observe that on ${\mathcal S_I}$, $D=\hat{D}=d$. Thus on ${\mathcal S_I}$, 
\[
D^2-y^\top y+2y^\top\xh= d^2-y^\top y+2x^\top y =x^\top x.
\]
Further, from (\ref{z1}), (\ref{z2}) and (\ref{z3}) on ${\mathcal S_I}$ 
\begin{eqnarray*}
\begin{split}
\dot{z}_1(t)-\dot{z}_2(t)+\dot{z}_3^\top(t)\xh(t)&= -\aa\left (z_1(t)-z_2(t)+z_3^\top(t)\xh(t)\right ) \\
&+\frac{D^2(t)-y^\top(t)y(t)+2y^\top(t)\xh(t)}{2}\\
&=-\aa\left (z_1(t)-z_2(t)+z_3^\top(t)x\right )\\ &+\frac{x^\top x}{2}
\end{split}
\end{eqnarray*}
Consequently, on ${\mathcal S_I}$; $\dot{z}_1(t)-\dot{z}_2(t)+\dot{z}_3^\top(t)\xh(t)=0$,
i.e. $\dot{z}_1(t)-\dot{z}_2(t)+\dot{z}_3^\top(t)x=0$.
Then,
\[
\eta (t)- m(t) +V^\top(t)\xh(t)=\dot{z}_1(t)-\dot{z}_2(t)+\dot{z}_3^\top(t)\xh(t)=0,
\]
and so $\dot{\xh}(t)=0$ from (\ref{aloc}).  Last because of (\ref{dd}),  and the fact that $\hat{D}=d$, we have that $\dot{\hat{D}}(t)=0$.
\end{proof}
Then we have the following theorem whose proof is given in Appendix \ref{app:proof1}.

\begin{theorem} \la{lass}
Consider the system defined in (\ref{D}) -  (\ref{claw}) and
${\mathcal S_I}$ defined in Lemma \ref{linv}. Then for arbitrary initial conditions,
subject to
the requirement that $\hat{D}(0)>0$ and Assumption \ref{aa}, $\|\xh^\top(t) , y^\top(t), z_1(t), z_2(t), z_3^\top(t)\|$ is bounded $\forall t\geq0$, and there holds:
\be \la{gc}
\lim_{t\rightarrow \infty} \min_{z\in {\mathcal S_I} } \|z-[\xh^\top(t) , y^\top(t), z_1(t), z_2(t), z_3^\top(t)]^\top \|=0.
\ee
Further, convergence is uniform in the initial time. 
\end{theorem}
\begin{remark}
Theorem \ref{lass} in particular implies that, $\tilde{x}(t)$ is bounded, $\lim_{t\rightarrow \infty} \xt (t)=0$, $\lim_{t\rightarrow \infty} D (t)=d$, and $\lim_{t\rightarrow \infty} \left ( y(t)-x- y^*(t) \right )=0$, where $y^*(t) $ is a nonzero solution of (\ref{y*}) with nonzero $y^*(0)$ that for all $t$ obeys $\|y^*(t)\|=d$. The first two limits follow from the definition of $ {\mathcal S_I} $. The last limit is a consequence of (\ref{claw}).
\end{remark} 
\setcounter{equation}{0}
\section{Exponential convergence and Persistent Excitation} \la{pe}
Having shown uniform  asymptotic stability in the foregoing, we will now demonstrate that in fact the stability is exponential. For this we establish a persistent spanning condition first on $\dot{y}(t)$, and ultimately on $V(t)$. This will be the key to establishing the \emph{exponential} convergence of $\hat{x}$ to $x$, which is a strengthening of the result of Theorem \ref{lass}. First we establish certain conditions on the state transition matrix for $A(t)$, i.e. on  $  \Phi(t,t_0) $ that obeys for
all $t,t_0$,
\be \la{st}
\dot{\Phi}(t,t_0) =A(t) \Phi(t,t_0).
\ee
\begin{lemma}\la{lorth}
Consider $\Phi(t,t_0)$ defined in (\ref{st}),with Assumption \ref{aa}. Then for all $t,t_0$
\[
\Phi^\top(t,t_0) \Phi(t,t_0) =I.
\]
\end{lemma}
\begin{proof}
Consider (\ref{y*}). Then for all $t,t_0$, and $y^*(t_0)$
\[
y^*(t)=\Phi(t,t_0) y^*(t_0).
\]
Further because of (\ref{skew}),
\[
\|y^*(t)\|=\|y^*(t_0)\|.
\]
Since this holds for all $ y^*(t_0)$, the result holds.
\end{proof}
Next the first promised result.
\begin{lemma} \la{lydpe}
Consider the system defined in (\ref{D}) - (\ref{claw}),  subject to
the requirement that $\hat{D}(0)>0$ and Assumption \ref{aa}. Then there exists a $T_2$ such that
for all $t\geq 0$,
\be \la{ydpe}
\frac{\aa_1 d^2}{2} I \leq \int_{t}^{t+T_2} \dot{y}(\tau)\dot{y}(\tau)^\top d \tau \leq \alpha_2 I
\ee
\end{lemma}
\begin{proof} See Appendix \ref{app:proof2}
\end{proof}
We now show that $V(t)$ satisfies a p.e. condition as well.
\begin{theorem} \la{tpe}
Consider the system defined in (\ref{D}) - (\ref{claw}),  subject to
the requirement that $\hat{D}(0)>0$ and Assumption \ref{aa}. Then there exist 
$\alpha_{3}>0,\alpha_{4}>0,T_3>0$ such that for all $t\ge 0$
\begin{equation}\label{pev}
\alpha_{3}I\le \int \limits_{t}^{t+T_3} V(\tau)V^{\top}(\tau)d\tau \le \alpha_{4}I.
\end{equation}
\end{theorem}
\begin{proof}
Follows directly Lemma \ref{lydpe} and \cite{DandachFidanDasguptaAnderson_SCL_08}. If we permitted $A(t)$ to lose differentiability on a countable set then
the marriage of \cite{DandachFidanDasguptaAnderson_SCL_08} with techniques developed in \cite{DAT} will provide a proof.
\end{proof}
We are now ready to prove exponential converegence of $\xh$ to $x$. Then in view of Lemma \ref{L1}, and fact that
$y(t)$ is bounded, $D(t)$ converges 
exponentially to $d$.
Define $p(t)= \eta (t)- m(t) +V^\top(t)x$, and observe from (\ref{exp}) that $\dot{p}(t)=-\aa p(t)$.
Observe also that $\eta (t)- m(t) +V^\top(t)\xh(t)= p(t)+V^\top(t)\xt (t)$. Thus one obtains:
\be \la{svr}
\left [ \begin{array}{c} \dot{\xt}(t)\\ \dot{p}(t) \end{array} \right ]
=\left [ \begin{array}{cc} -\gamma V(t)V^\top(t) & -\gamma V(t) \\ 0 & -\aa\end{array} \right ]
\left [ \begin{array}{c} \xt(t)\\ p(t) \end{array} \right ].
\ee
\begin{theorem} \la{tmain}
Consider the system defined in (\ref{D}) -  (\ref{claw}),  subject to
the requirement that $\hat{D}(0)>0$ and Assumption \ref{aa}. Then (\ref{svr}) is eas.
\end{theorem}
\begin{proof}
From Theorem \ref{tpe}, (\ref{pev}) holds. Consequently, from
\cite{bda77}, the system  
\[
\dot{\zeta}=-\gamma VV^\top \zeta
\]
is eas. Then the triangular nature of (\ref{svr}) establishes the result.
\end{proof}
\setcounter{equation}{0}
\section{Choosing $A(t)$} \la{sA}
In this section we focus on selection of $A(t)$ to satisfy Assumption \ref{aa}. Consider first $n=2$, we show that with $E$ as in (\ref{E}) the matrix $A(t)=aE$ obeys
the requirements of assumption \ref{aa}. Indeed consider the Lemma below.

\begin{lemma} \la{ls2s}
With $E$ as in (\ref{E}), and a real scalar nonzero $a$ consider \be \la{Ea}
\dot{\xi}(t)=aE\xi(t),
\ee
with $\xi:\mR\rightarrow \mR^2$.  Denote  $\xi=[\xi_1,\xi_2]^\top$. Define $\beta$ as 
\be \la{bb}
\bb(t_0) =\angle (\xi_1+\boldsymbol{i}\xi_2).
\ee
i.e. the argument of the complex number $\xi_1+\boldsymbol{i}\xi_2$. Then there holds for all $t\geq t_0\geq 0$,
\be \la{yss}
\xi(t)= \|\xi(t_0) \| [\cos(a(t-t_0)+\bb(t_0)), \sin(a(t-t_0)+\bb(t_0))]^\top.
\ee
\end{lemma}
\begin{proof}
Follows from the facts that $\xi(t_0)=\|\xi(t_0) \|[\cos(\bb(t_0)), \sin(\bb(t_0))]^\top$, and that 
the state transition matrix corresponding to (\ref{Ea}) is:
\[
e^{aEt}=\left[ \begin{array}{ccc} \cos at & & -\sin at\\\sin at & & \cos at \end{array}\right ].
\]
\end{proof}
The fact that (\ref{yss}) satisfies (\ref{ps})
with $y^*$ identified with $\xi$, is trivial to check. It is also clear that under this selection, $y(t)$ 
circumnavigates $x$ with an angular speed of $|a|$. 

In preparation for treating the $n=3$ case, 
we make the following observations.
\begin{lemma} \la{lp}
Consider (\ref{yss}). Suppose for any $t_0$, all $t\in \left [t_0,t_0+ \frac{\pi}{2|a|} \right]$,  some $\theta \in \mR^2$, there exists $\epsilon$ such that $|\theta^\top\dot{\xi}(t)| \leq \epsilon \|\theta \|$. Then 
\be \la{eg}
\epsilon \geq |a|\|\xi(t_0)\|\|\theta \|.
\ee
Further with $\xi=[\xi_1,\xi_2]^\top$, for all $t_0$ and $i\in\{1,2\}$,
\be \la{yabs}
\left | \xi_i\left (t_0 +   \frac{\pi}{2|a|} \right ) - \xi_i
\left (t_0  \right ) \right |  = \|\xi(t_0)\|.
\ee
\end{lemma}
\begin{proof} For some real $\psi$ there holds $\theta =\|\theta \|[\sin\psi,-\cos\psi]^\top$.
Hence, $E^\top\theta =\|\theta \|[\cos\psi,\sin\psi]^\top$. Thus under (\ref{yss}), 
\begin{equation}\label{eq:ztheta}
\theta^\top E \xi(t)= \|\xi(t_0)\|\|\theta \|\cos(a(t-t_0)+\bb(t_0)-\psi).
\end{equation}
Therefore, on any interval $[t_0,t_0+\pi/2|a|]$ the maximum of $|\theta^\top \dot{\xi}|=|a\theta^\top E\xi|$ is  $|a|||\xi(t_0)||||\theta||$. Further (\ref{yabs}) is a direct consequence of (\ref{yss}).
\end{proof}
Finally we prove the following lemma.
\begin{lemma} \la{lrho}
Consider 
\be \la{zf}
\dot{\xi}(t)=af(t)E\xi(t),
\ee
where $f:\mR\rightarrow\mR$ and $|f(t)| \leq 1\;\;\; \forall t$.
Then for all $0\leq t_0 \leq t$, $\| \xi(t)- \xi(t_0)\|  \leq (t-t_0)|a|  \left\|z(t_0)\right\|$.
\end{lemma}
\begin{proof}
Under  (\ref{zf}), for all $t\geq t_0$, $\|\xi(t)\|=\|\xi(t_0)\|$. Thus, 
\begin{eqnarray*}
\| \xi(t)-\xi(t_0)\|  &=&
 \| \int_{t_0}^{t}  f(\tau) a E \xi(\tau) d\tau \|\\
&\leq&  (t-t_0)|a| \left\|\xi(t_0)\right\|
\end{eqnarray*}
\end{proof}
To address the $n=3$ case we first preclude the possibility that $A(t)$ can be a constant matrix. Indeed
observe that no skew-symmetric matrix in $\mR^{3\times 3}$ can be nonsingular, as if $\lambda$ is an
eigenvalue of a skew symmetric matrix then so is $-\lambda$. Thus for any odd $n$, an $n\times n$ skew symmetric matrix must
have a zero eigenvalue. To complete the argument we present the following Lemma.
\begin{lemma} \la{noconst}
Suppose in (\ref{y*}) $A(t)\equiv A$ for all $t$ and $A$ is singular. Then (\ref{ps}) cannot hold.
\end{lemma}
\begin{proof}
If $A$ is singular, then $e^{At}$ has an eigenvalue
at one. Thus there exists a $y^*(0)$ such that for all $t$, $y^*(t)=e^{At}y^*(0)$ is a constant, i.e. for this $y^*(0)$,
$\dot{y}^*\equiv 0$.
\end{proof}
Thus, we must search for a periodic $A(t)$ to meet the requirements of Assumption \ref{aa}. Effectively, the $A(t)$
 we will choose will
switch periodically between the two $3\times 3$ matrices
\be \la{B}
B=0\oplus (bE),\mbox{ and}
\ee
\be \la{C}
C=(cE)\oplus 0
\ee
$b$ and $c$ being real nonzero scalars, and $\oplus$ denoting direct sum. However to ensure that the resulting matrix is
differentiable, we require
a differentiable transition between $B$ and $C$. To achieve this define a nondecreasing $g:\mR\rightarrow\mR$, that obeys:
\be \la{g1}
 g(t)=0 \;\;\;\forall \; t\leq 0 
\ee
\be \la{g2}
 g(t)=1, \;\;\;\forall \;  t\geq 1, \mbox{ and}
\ee
\be \la{g3}
 g(t)\mbox{ is twice differentiable } \forall t.
\ee
An example of such a $g(t)$ is
\be \la{gt}
g(t)=\left \{ \begin{array}{ll} \frac{1}{2} \left ( 1-\cos\left (\pi t \right ) \right ) & 0\leq t \leq 1\\
0& t< 0\\
1 & t>1. \end{array} \right .
\ee
Clearly this satisfies (\ref{g1}) and (\ref{g2}). Further (\ref{g3}) holds as
\[
\lim_{t\rightarrow 0_+}\dot{g}(t) = \lim_{t\rightarrow 1_-}\dot{g}(t)=0.
\]
Now, for nonzero scalars $b$ and $c$, we will select $A(t)$ as follows. For a suitably small $\rho >0$,
define
\be \la{T13}
\bar{T}_1= \rho, \;\; \bar{T}_2=  \rho   + \frac{\pi}{|b|}, \;\; \bar{T}_3=  2\rho   + \frac{\pi}{|b|},
\ee
and
 \be \la{T46}
\bar{T}_4=  3\rho   + \frac{\pi}{|b|}, \;\; \bar{T}_5=  3\rho   + \frac{\pi}{|b|} +\frac{\pi}{|c|},  \;\; 
T=\bar{T}_6
=4\rho   + \frac{\pi}{|b|} +\frac{\pi}{|c|}.  \ee
For all $t$, let $K_T(t)$ denote the largest integer $k$ satisfying $t\geq kT$ and let $r_T(t)=t-K_T(t)T$. Then define $A(t)$ as
\be \la{a3}
A(t)=\left \{ \begin{array}{ll} g\left (\frac{t}{\rho} \right ) B & 0\leq r_T(t) \leq \bar{T}_1 \\
B & \bar{T}_1 \leq r_T(t) \leq \bar{T}_2 \\
\left (1- g\left (\frac{t-\bar{T}_2  }{\rho} \right )\right) B\;\;\;&  \bar{T}_2\leq r_T(t)\leq \bar{T}_3\\
g\left (\frac{t-\bar{T}_3  }{\rho} \right ) C\;\;\;&  \bar{T}_3\leq r_T(t)\leq \bar{T}_4\\
C & \bar{T}_4 \leq r_T(t) \leq \bar{T}_5 \\
\left (1- g\left (\frac{t-\bar{T}_5  }{\rho} \right )\right) C\;\;\;&  \bar{T}_5\leq r_T(t)\leq \bar{T}_6=T\\
\end{array} \right .
\ee
Observe that (\ref{a3}) automatically satisfies (i-iii) of Assumption \ref{aa}. To show that it satisfies (iv) as well, we present the following result from \cite{mit}.
\begin{lemma} \la{Lmit}
Suppose on a closed interval $ {\mathcal{I}} \subset \mR $ of length $\Omega$, a signal 
$w: {\mathcal{I}} \rightarrow \mR  $ is twice differentiable and for some $\eps_1$ and $M'$
\[
|w(t)| \leq \eps _1 \mbox{ and } |\ddot{w}(t)|\leq M'\;\;\; \forall \;\;\; t\in {\mathcal{I}}.
\]
Then for some $M$ independent of $\eps_1$,  ${\mathcal{I}}$ and  $M'$, and $M'' = \max (M', 2 \epsilon_1 \Omega ^{-2})$ one has: 
\[
|\dot{w}(t)| \leq  M(M''\eps_1)^{1/2}\;\;\; \forall \;\;\; t\in {\mathcal{I}}.
\]
\end{lemma}
Next, we establish the following result.
\begin{theorem} \la{t3pers}
Consider (\ref{y*}) with $A(t)$ defined in (\ref{T13})-(\ref{a3}). Then for every pair of nonzero $b,c$
there
exists a $\rho^*$ such that 
(\ref{ps}) holds
 for all $0<\rho\leq \rho^*$.
\end{theorem}
\begin{proof}
See Appendix \ref{app:proof3}.
\end{proof}
\setcounter{equation}{0}
\section{Simulations}\label{sec:sim}
In this case we study the behaviour of the system in four different scenarios in 2-dimensional space.

In the first simulation we study the case where $x=[0.5,\; 3]^\top$, $d=2$, and $y(0)=[8,\;5]^\top$. The corresponding result is depicted in Fig. \ref{fig:withoutnorm}. A closer look at the agent trajectory reveals a very small radius turn near the point $[2,\;1]^\top$. The reason for this behaviour is the following. The term $(\hat{D}^2(t)-d^2)(y(t)-\hat{x}(t)$ in (\ref{claw}) is designed to force $y(t)$ to move on a straight line trajectory  in a manner that drives $\hat{D}$ to $d$. The second term $A(t)(y(t)-\hat{x}(t))$ forces $y(t)$ to rotate around $\hat{x}(t)$. Initially the first term is dominant, and the agent quickly travels a long distance on  an almost straight line. By the time the agent reaches $[2,1]^T$, the rotational motion component becomes comparable to the straight line motion component; hence the effect of this change shows itself as a sharp turn. To make the trajectory smoother, in the second scenario we consider the same setting with the difference that instead of using $\dot{y}(t)$, we use the normalized signal; in other words the agent is moving with constant speed. We use the normalized version of (\ref{claw}). Furhermore,
\begin{equation*}
\dot{\overline{y}}(t)= \dot{\hat{x}}(t) -\left [ (\hat{D}^2(t)-d^2) I- A(t)\right ](y(t)-\hat{x}(t)),
\end{equation*}
\begin{equation*}
\dot{y}(t)=\left \{ \begin{array}{ll}
\dot{\overline{y}}(t)/\|\dot{\overline{y}}(t)\|&	\|\dot{\overline{y}}(t)\|\neq 0\\
0	&	\|\dot{\overline{y}}(t)\|=0
\end{array} \right.
\end{equation*}
 As can be observed, the small radius turn is replaced by one of larger radius. The result is presented in Fig. \ref{fig:withnorm}. In the first two cases the  desired  orbit at the prescribed distance is achieved. In the third simulation we studied the behaviour of the system when the source slowly drifts on a circle. on a circle with angular velocity equal to $0.005$. See Fig. \ref{fig:withdrift}. The agent maintain its distance from the source in a neighbourhood of the desired distance. Notice that the speed of the source is always much less than the speed of the agent. In the last simulation we consider the case where the distance measurement is noisy, and it is assumed that $\mathop{ln}\bar{D}=\mathop{ln}D+\mu(t)$, where $\bar{D}$ is measurement and $\mu$ is a strict-sense stationary random process with $\mu(t)\sim N(0,\sigma^2)$, $\forall t$. The simulation result associated with this scenario is depicted in Fig. \ref{fig:withnoise}. As it can be observed the control law is still successful in moving the agent to an orbit with distance to the source kept close to its desired value. 
\begin{figure}
\begin{center}
\includegraphics[width=8 cm]{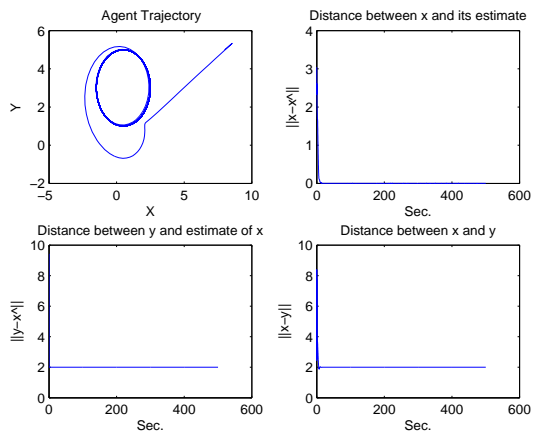}
\end{center}
\caption{Agent trajectory, agent distance from the estimate, agent distance from the real value, and distance between the estimate and true position of source.} \label{fig:withoutnorm}
\end{figure}

\begin{figure}
\begin{center}
\includegraphics[width=8 cm]{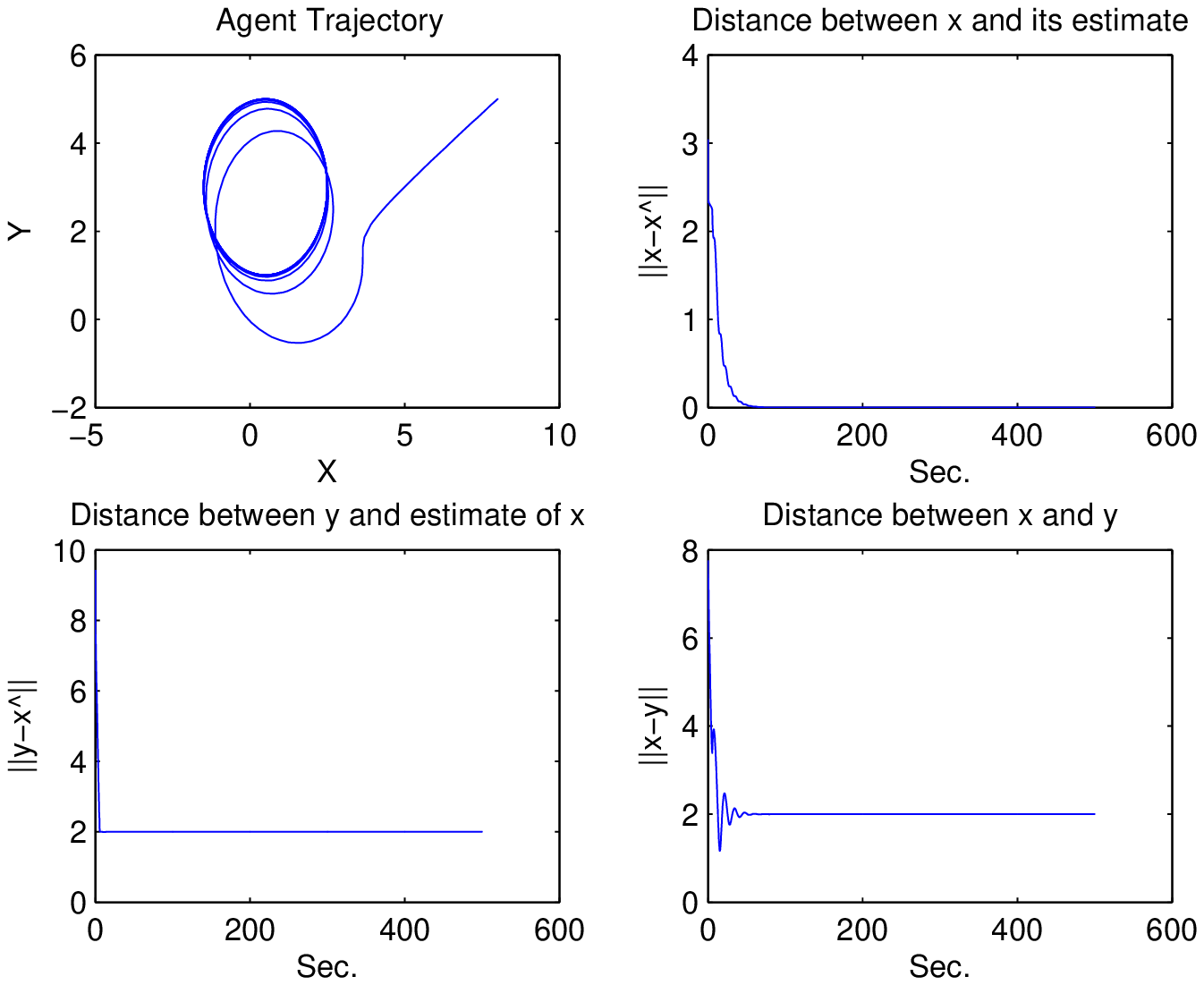}
\end{center}
\caption{Agent trajectory, agent distance from the estimate, agent distance from the real value, and distance between the estimate and true position of source, where agent is moving with constant speed.}\label{fig:withnorm}
\end{figure}

\begin{figure}
\begin{center}
\includegraphics[width=8 cm]{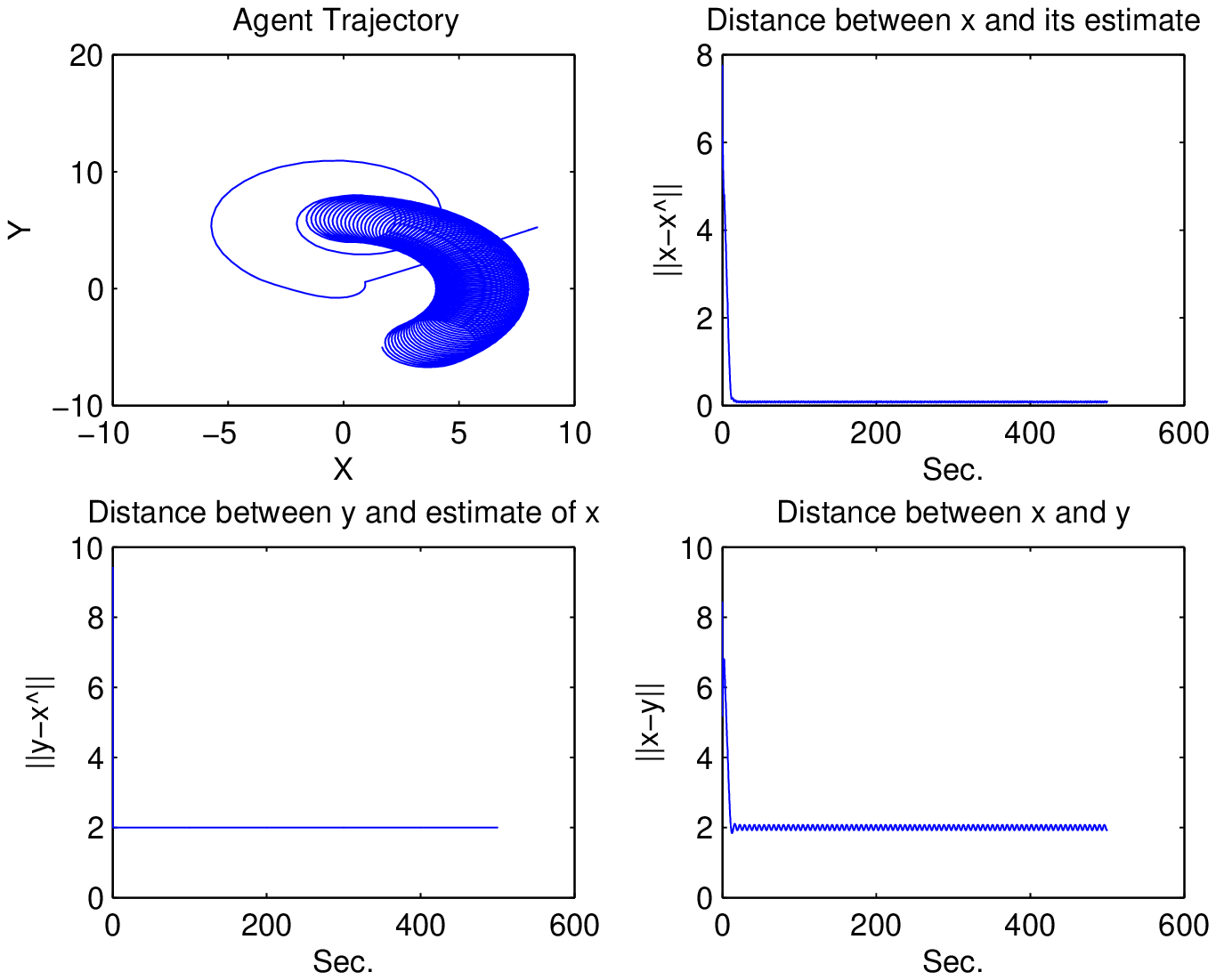}
\end{center}
\caption{Agent trajectory, agent distance from the estimate, agent distance from the real value, and distance between the estimate and true position of source, where the source is undergoing a drift.}\label{fig:withdrift}
\end{figure}

\begin{figure}
\begin{center}
\includegraphics[width=8 cm]{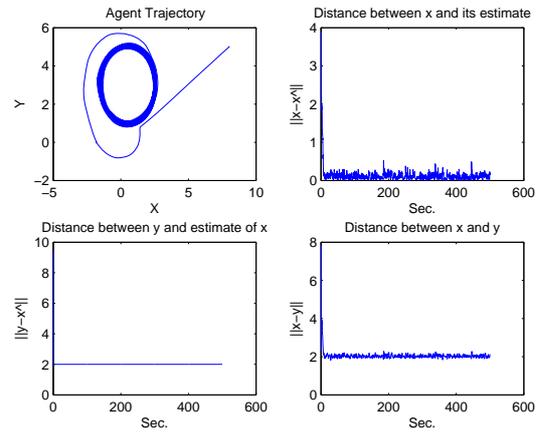}
\end{center}
\caption{Agent trajectory, agent distance from the estimate, agent distance from the real value, and distance between the estimate and true position of source, in the presence of noisy measurement.}\label{fig:withnoise}
\end{figure}
\section{Future Work and Concluding Remarks}
In this paper we proposed an algorithm to solve the problem of monitoring a source at an unknown position by a single agent while the only information available to the agent is its distance to the target. Stability of the system has been established. Furthermore, in simulations the performance of the method in the presence of noise and in the situations where the source is undergoing a drifting motion is presented. An important future work is to establish stability of the system when the source is undergoing a drifting motion. Another possible extension of the current scheme is to consider the cases where more than one agent is present.
\bibliographystyle{IEEEtran}
\bibliography{Adaptive}

\begin{thebibliography}{10}
\providecommand{\url}[1]{#1}
\csname url@rmstyle\endcsname
\providecommand{\newblock}{\relax}
\providecommand{\bibinfo}[2]{#2}
\providecommand\BIBentrySTDinterwordspacing{\spaceskip=0pt\relax}
\providecommand\BIBentryALTinterwordstretchfactor{4}
\providecommand\BIBentryALTinterwordspacing{\spaceskip=\fontdimen2\font plus
\BIBentryALTinterwordstretchfactor\fontdimen3\font minus
  \fontdimen4\font\relax}
\providecommand\BIBforeignlanguage[2]{{%
\expandafter\ifx\csname l@#1\endcsname\relax
\typeout{** WARNING: IEEEtran.bst: No hyphenation pattern has been}%
\typeout{** loaded for the language `#1'. Using the pattern for}%
\typeout{** the default language instead.}%
\else
\language=\csname l@#1\endcsname
\fi
#2}}

\bibitem{ShamesFidanAnderson_CDC_08}
I.~Shames, B.~Fidan, and B.~D.~O. Anderson, ``Close target reconnaissance using
  autonomous {UAV} formations,'' 2008, to appear in CDC08.

\bibitem{SinhaGhose_ACC_2005}
A.~Sinha and D.~Ghose, ``Generalization of the cyclic pursuit problem,'' in
  \emph{Proceedings of American Control Conference}, Portland, OR, 2005.

\bibitem{KimSugie_Auto_2007}
T.-H. Kim and T.~Sugie, ``Cooperative control for target-capturing task based
  on a cyclic pursuit strategy,'' \emph{Automatica}, vol.~43, pp. 1426--1431,
  2007.

\bibitem{CaoMorse_ECC_07}
M.~Cao and A.~S. Morse, ``Maintaining an autonomous agent's position in a
  moving formation with range-only measurements,'' in \emph{Proc. of European
  Control Conference}, pp. 3603--3608.

\bibitem{CaoMorse_ACC_07}
------, ``Station keeping in the plane with range-only measurements,'' in
  \emph{Proc. American Control Conference}, 2007, pp. 771--776.

\bibitem{DandachFidanDasguptaAnderson_SCL_08}
S.~Dandach, B.~Fidan, S.Dasgupta, and B.~D.~O. Anderson, ``A continuous time
  linear adaptive source localization algorithm robust to persistent drift,''
  2008, {System} and Control Letters, doi:10.1016/j.sysconle.2008.07.008.

\bibitem{Feldbaum60_Dual}
A.~A. Feldbaum, ``Dual control theory. {I-IV},'' {Automation} and Remote
  Control, vol. 21, 1960 and vol. 22, 1961.

\bibitem{DAT}
S.~Dasgupta, B.~D.~O. Anderson, and A.~C. Tsoi, ``Input conditions for
  continuous time adaptive systems problems,'' \emph{IEEE Transactions on
  Automatic Control}, vol.~35, pp. 78--82, January 1990.

\bibitem{bda77}
B.~D.~O. Anderson, ``Exponential stability of linear equations arising in
  adaptive identification,'' \emph{IEEE Transactions on Automatic Control}, pp.
  530--538, May-June 1977.

\bibitem{mit}
D.~S. Mitrinovic, \emph{Analytic Inequalities}.\hskip 1em plus 0.5em minus
  0.4em\relax Springer Verlag, 1970.

\end{thebibliography}

\appendix
\setcounter{equation}{0}
\renewcommand{\theequation}{A.\arabic{equation}}
\subsection{Proof of Theorem \ref{lass}} \label{app:proof1}
First observe that the system (\ref{D}) -  (\ref{claw}) is in fact {\it periodic}. Thus, convergence if it holds, will be uniform in
the initial time.

Consider $L(t)$ and $L_i(t)$ defined in Lemma \ref{lbd}.
Now for every finite initial conditions there is a $\Delta$ such that 
\[
[\xh^\top(0) , y^\top(0), z_1(0), z_2(0), z_3^\top(0)]^\top \in S(\Delta) .
\]
As shown above this set is compact.

Note that for all $i\in\{1,2,3\}$, $\dot{L}_i(t)=-2\alpha L_i(t)$.

Then because of Lemma \ref{L1}, (\ref{aloc}), (\ref{claw}) and (\ref{vr}), there holds
\begin{equation}\label{ldot}
\begin{split}
\dot{L}(t)&= -\frac{1}{2} \left ( \eta(t) - m(t) +V^\top(t) x \right )^2 \\
&-\xt^\top(t)V(t) \left ( \eta(t) - m(t) +V^\top(t) x +V^\top(t) \xt (t)\right )\\
&+    \left ( \hat{D}^2(t) -d^2 \right ) (y^\top(t)-\xh^\top (t))\left (\dot{y}(t) - \dot{\xh}(t) \right  )\\
& -\aa \sum_{i=1}^3L_i(t) \\
&= -\frac{1}{2} \left ( \eta(t) - m(t) +V^\top(t) \xh(t) \right )^2-\frac{1}{2}(\xt^\top(t)V(t))^2\\
&-    \left ( \hat{D}^2(t) -d^2 \right ) ^2\hat{D}^2(t) -\aa \sum_{i=1}^3L_i(t) \leq 0 
\end{split}
\end{equation}
Thus all trajectories commencing in $S(\Delta)$ lie in $S(\Delta)$, and hence $\|\hat{x}^\top(t),\;y^\top(t),\;z_1(t),\;z_2(t),\;z_3^\top(t)\|$ is bounded $\forall t\geq 0$.

To prove the theorem we simply need to show that convergence occurs to $S(\Delta)\bigcap {\mathcal S_I}$. By Lasalle's theorem this will hold if
$S(\Delta)\bigcap {\mathcal S_I}$ is the largest invariant set in $S(\Delta)$. Lemma \ref{linv}
and (\ref{ldot})  already shows that 
$S(\Delta)\bigcap {\mathcal S_I}$ is an invariant set. To show that it is the largest invariant set 
we must show that
\be \la{ident}
\dot{L} \equiv 0,
\ee
implies that there hold for all $t\geq 0$
\be \la{xc}
\xh(t)=x,
\ee
\be \la{dc}
\|y(t)-x\|=d,
\ee
\be \la{zc}
z_1(t)-z_2(t)+z_3^\top(t)x = \frac{x^\top x}{2\aa}.
\ee
Now (\ref{ident}) necessitates the pair of identities:
\be \la{xhc}
\dot{\xh} \equiv 0,
\ee
\be \la{dhc}
\hat{D} \equiv d.
\ee
From (\ref{xhc}) one obtains that for some constant  $x^*$
\be \la{xstead}
\xh\equiv x^*.
\ee
Further under (\ref{dhc}), (\ref{xstead}) and (\ref{xhc}), (\ref{claw}) reduces to
\be \la{ystead}
\dot{y}(t)= A(t)(y(t)-x^*).
\ee
Thus, along trajectories corresponding to (\ref{ident}), there holds:
\be \la{sinus}
y(t)=x^*+ y^*(t),
\ee
where $y^*(t)$ is a solution of (\ref{y*}).
Denoting $a=\eta (t_0)- m(t_0) +V^\top(t_0)x$, because of (\ref{aloc}), (\ref{xhc}), (\ref{xstead}),  and Lemma \ref{lem:Lemma1} for all $t\geq t_0$, one has that:
\begin{eqnarray*}
0&=& \eta(t) - m(t) +V^\top(t) x^*\\
&=&\eta(t) - m(t) +V^\top(t) x +V^\top(t)(x^*-x)\\
&=& ae^{-\aa (t-t_0)} + V^\top(t)(x^*-x).
\end{eqnarray*}
Thus one has
\begin{equation}\label{eq:V1}
V^\top(t)(x^*-x)= -ae^{-\aa (t-t_0)} 
\end{equation}
\begin{equation}\label{eq:dotV1}
\frac{d}{dt} \left \{   V^\top(t) (x^*-x) \right \}=\aa a e^{-\aa(t-t_0)}. 
\end{equation}
As $(x^*-x)$ is a constant, combining (\ref{eq:V1}), (\ref{eq:dotV1}) with (\ref{n3}), we obtain that along trajectories corresponding to (\ref{ident}), for all $t\geq t_0$, 
\begin{eqnarray*}
\aa a e^{-\aa(t-t_0)} = \aa a e^{-\aa(t-t_0)} + \dot{y}^\top(t)(x^*-x) ,
\end{eqnarray*}
i.e. for all $t\geq 0$, $ \dot{y}^\top(t)(x^*-x) =0$. Given that $(x^*-x) $ is a constant and because of (\ref{ps}) and (\ref{sinus}) this can only hold if $x^*=x$.
Further in view of (\ref{dhc}) along (\ref{ident}) $D\equiv d$.

Thus (\ref{xc}) and (\ref{dc}) are necessary for (\ref{ident}) to hold.
Finally we observe from (\ref{ident}), (\ref{xc}), and the first term in (\ref{ldot}) that:
\begin{eqnarray*}
\begin{split}
&  \eta -m +V^\top\xh \equiv \eta -m +V^\top x \equiv 0, \\
\Rightarrow& \dot{\eta} -\dot{m} +\dot{V}^\top x \equiv 0 \\
\Leftrightarrow& -\aa(z_1-z_2+z_3^\top x) + \frac{ D^2 - \|y\|^2 +2y^\top x}{2}\equiv 0, \\
\Leftrightarrow& -\aa \left ( z_1-z_2+z_3^\top x - \frac{ x^\top x}{2\aa} \right )\equiv 0, \mbox{ as $D\equiv d$}, 
\end{split}\end{eqnarray*}
establishing (\ref{zc}).

\setcounter{equation}{0}
\renewcommand{\theequation}{B.\arabic{equation}}
\subsection{Proof of Lemma \ref{lydpe}}\label{app:proof2}
A consequence of assumption \ref{aa} is that for all unit $\theta\in \mR^n$,  $t\geq 0$ and $z\in \mR^n$, there holds:
\be \la{psm}
\aa_1  \| z\|^2 \leq \int_{t}^{t+T_1} | \theta^\top A(\tau)\Phi(\tau,t) z |^2 d \tau .
\ee
Further because of (\ref{claw}) for all $t_1\geq 0$ and $t\geq t_1$ there holds:
\be \la{clawc}
\begin{split}
\dot{y}(t)-\dot{\xh}(t)&= A(t)\Phi(t,t_1) (y(t_1)-\xh(t_1) )\\
& -A(t) \int_{t_1}^t \Phi(t,\tau) ( \hat{D}^2(\tau)-d^2 )
  (y(\tau)-\xh(\tau) ) d\tau .
\end{split}
\ee
Assumption \ref{aa} ensures that $A(t)$ is bounded. Thus there exists $M_2$ 
such that
\be \la{abd}
\|A(t)\| \leq M_2\;\;\;\;\forall t.
\ee
 Further because of Lemma \ref{L1} and Theorem \ref{lass}, there is a 
$\gamma>0$, such that for every $\epsilon>0$, there is a $t_2$ such that for all $t\geq t_2$,
\be \la{deps}
|d-\|y(t)-\xh(t)\|| \leq \epsilon,
\ee
\be \la{xdps}
\|\dot{\xh}(t)\| \leq \epsilon,
\ee
\be \la{ddeps}
|\hat{D}^2(t)-d^2 |\leq \epsilon e^{-\gamma (t-t_2)}.
\ee
Thus because of Lemma \ref{lorth}, (\ref{clawc})- (\ref{ddeps}) 
 for every  unit $\theta\in \mR^n$ and  $t\geq t_2$ 
\be \la{claw2}
|\theta^\top\dot{y}(t)|\geq |\theta^\top A(t)\Phi(t,t_2) (y(t_2)-\xh(t_2) ) | -\epsilon-
 \frac{\epsilon (d+\epsilon)M_2}{\gamma}.  \nonumber\\
\ee
Thus there exist $K_i$ all positive such that 
for all $t\geq t_2$ there holds
\be \la{claw3}
|\theta^\top\dot{y}(t)|^2\geq |\theta^\top A(t)\Phi(t,t_2) (y(t_2)-\xh(t_2) ) |^2 - \sum_{i=1}^4K_i\epsilon^i .
\ee
Choose $T_2=T_1+t_2$. Then because of (\ref{psm}) and (\ref{claw3}) for all $t>0$,
there holds:
\begin{eqnarray*} 
\begin{array}{l}
\int_{t}^{t+T_2} |\theta^\top\dot{y}(\tau)|^2d\tau\\
\geq \int_{t+t_2}^{t+T_2} |\theta^\top\dot{y}(\tau)|^2d\tau \\
\geq \int_{t+t_2}^{t+T_2} |\theta^\top A(\tau)\Phi(\tau,t_2) (y(t_2)-\xh(t_2) )  |^2 d\tau  \\- T_1 \sum_{i=1}^4K_i\epsilon^i\\
\geq \aa_1 (d- \epsilon )^2- T_1 \sum_{i=1}^4K_i\epsilon^i\\
\geq \aa_1 d^2 - 2\aa_1 d\epsilon - T_1 \sum_{i=1}^4K_i\epsilon^i
\end{array}
\end{eqnarray*}
Then the left inequality in (\ref{ydpe}) follows by choosing $\epsilon$ so that
\be \la{epsbd}
T_1 \sum_{i=1}^4K_i\epsilon^i+2\aa_1 d\epsilon \leq \aa_1 d^2/2.
\ee
The  right inequality in (\ref{ydpe}) follows from the boundedness of $\hat{x}$, (\ref{claw}), (\ref{y*}), (\ref{ps}), and Lemma  \ref{lem:Lemma1}.

\setcounter{equation}{0}
\renewcommand{\theequation}{C.\arabic{equation}}
\subsection{Proof of Theorem \ref{t3pers}}\label{app:proof3}
We will prove the result by contradiction. First observe that as $A(t)$ is differentiable and $\ddot{y}^*$ is bounded.
Also observe that if (\ref{ps}) holds for $\|y^*(0)\|=1$, then it holds for
arbitrary $\|y^*(0)\|$.
Thus assume that  $\|y^*(0)\|=1$. Consequently for all $t\geq 0$
\be \la{un}
\|y^*(t)\|=1. 
\ee
Suppose  (\ref{ps}) is violated. Then for all $\eps _2  >0$ and $T_3>0$, there exists a $t_0$
and a unit norm  $\theta = [\theta_1, \theta_2 , \theta_3]^\top \in \mR^3$, such that 
\[
\int_{t_0}^{t_0+T_3} (\theta^\top \dot{y}^*(\tau) )^2d\tau \leq \eps _2 ^2.
\]
Thus from Lemma \ref{Lmit} for some $M_3 $, all $\eps _2  >0$, some $T_4(\eps _2 )$, dependent only on the bound on 
$\ddot{y}^*(\cdot)$ and $\eps _2 $, and all $T_3> T_4(\eps _2 ) $, 
there exists a $t_0$  and unit norm  $\theta \in \mR^3$, for which 
\be \la{ydsm}
\left | \theta^\top \dot{y}^*(t)\right | \leq M_3 \eps _2 ^{1/2} \;\;\; \forall t\in [t_0,t_0+T_3] .
\ee
Choose 
\be \la{t1}
t_1 =\min\{kT \geq t_0+ T_4(\eps _2 ) |  k\in {\mathbb{Z}}_+\}.
\ee
Denote $ y^* =[y^*_1, y^* _2,y^* _3]^\top$. Observe at least one
of $\|[\theta_1, \theta_2]^\top\| $ or $\|[\theta_2, \theta_3]^\top\| $ must exceed $ 1/\sqrt{3} $, since $\theta$ has unit norm.
We consider two cases.

\noindent
{\bf Case I: $\| [\theta_1, \theta_2]^\top \| >1/\sqrt{3} $ .} 

Since the inequality in (\ref{ydsm}) holds on the indicated interval, it must hold for all $t\in[t_1+kT+\bar{T}_4, t_1+kT+\bar{T}_5]$,
$k\in {\mathbb{Z}}$. 

Thus for all $t\in[t_1+kT+\bar{T}_4, t_1+kT+\bar{T}_5]$ and 
$k\in {\mathbb{Z}}$, there holds:
\be \la{y12b1}
|[\theta_1, \theta_2]^\top [ \dot{y}^*_1(t) , \dot{y}^*_2(t)]|\leq  M_3 \eps _2 ^{1/2} .
\ee

Now for all $t\in[t_1+kT+\bar{T}_4, t_1+kT+\bar{T}_5]$,
there also holds:
\[
\left [ \begin{array}{c} \dot{y}^*_1(t) \\ \dot{y}^*_2(t) \end{array}\right ]
=cE\left [ \begin{array}{c} y^*_1(t) \\ y^*_2(t) \end{array}\right ].
\]
Thus, from (\ref{eg}) of Lemma \ref{lp} and the hypothesis of the case, we obtain
that for all $k\in {\mathbb{Z}}$,
\be\la{12bd}
\|[y^*_1(t_1+kT+\bar{T}_4) , y^*_2(t_1+kT+\bar{T}_4)]^\top \| \leq \frac{ \sqrt{3} M_3}{|c|}  \eps _2 ^{1/2} .
\ee
Further with some $h_1:\mR\rightarrow \mR$, in the interval $[kT+\bar{T}_4, (k+1)T]$,$ \left [ \begin{array}{c} \dot{y}^*_1(t) \\ \dot{y}^*_2(t) \end{array}\right ]
=h_1(t)E\left [ \begin{array}{c} y^*_1(t) \\ y^*_2(t) \end{array}\right ]$. Thus,
\be \la{y12b}
\begin{array}{l}
\|[y^*_1(t_1+(k+1)T) , y^*_2(t_1+(k+1)T) ]^\top\|=\\
\|[y^*_1(t_1+kT+\bar{T}_4) , y^*_2(t_1+kT+\bar{T}_4)]^\top \| \leq \frac{\sqrt{3} M_3}{|c|} 
 \eps _2 ^{1/2} .
\end{array}
\ee
Consequently because of (\ref{un}), there holds: 
\be \la{y23b}
\begin{array}{l}
\|[y^*_2(t_1+(k+1)T) , y^*_3(t_1+(k+1)T)]^\top \|\geq\\
| y^*_3(t_1+(k+1)T)|\geq 1- \frac{\sqrt{3}M_3}{|c|}  \eps _2 ^{1/2} 
\end{array}
\ee
Further throughout the interval $t\in [t_1+kT,t_1+kT+T_3]$ for some $h_2:\mR\rightarrow \mR$, $|h_2(t)|\leq 1$, 
\be \la{d23}
\left [ \begin{array}{c} \dot{y}^*_2(t) \\ \dot{y}^*_3(t) \end{array}\right ]
=h_2(t)E\left [ \begin{array}{c} y^*_2(t) \\ y^*_3(t) \end{array}\right ].
\ee
Thus from Lemma \ref{lrho} and (\ref{y12b})
\be \la{y2b}
 |y^*_2(t_1+kT+\bar{T}_1)|\leq \frac{\sqrt{3} M_3}{|c|}  \eps _2 ^{1/2} + \rho |b|.
\ee
Also from (\ref{d23}) and (\ref{y23b})
\[
\|[y^*_2(t) , y^*_3(t)]^\top \| \geq 1- \frac{\sqrt{3} M_3}{|c|}  \eps _2 ^{1/2} 
\]
holds for all $t\in [t_1+kT,t_1+kT+\bar{T}_3]$.
Notice in the interval $[[t_1+kT+\bar{T}_1,t_1+kT+\bar{T}_2]$, (\ref{d23}) holds with $h_2(t)=b$.
Thus from (\ref{yabs}) of Lemma \ref{lp},
\be \la{inter}
\begin{array}{l}
|  y^*_2(t_1+kT+\bar{T}_2)| +| y^*_2(t_1+kT+\bar{T}_1)|\\
\geq|  y^*_2(t_1+T+\bar{T}_2) -  y^*_2(t_1+T+\bar{T}_1)| \geq 1- \frac{\sqrt{3} M_3}{|c|}  \eps _2 ^{1/2} 
\end{array}
\ee
Consequently, from (\ref{y2b})
\[
 |y^*_2(t_1+kT+\bar{T}_2)|\geq 1- \frac{2\sqrt{3} M_3}{|c|}  \eps _2 ^{1/2}  -\rho |b| .
\]
Further, from Lemma \ref{lrho}
\begin{equation}\label{eq:star}
 |y^*_2(t_1+kT+\bar{T}_4)| \geq 1- \frac{2\sqrt{3} M_3}{|c|}  \eps _2 ^{1/2} -  \rho 
\left (2|b| +|c|\right ).
\end{equation}
Then for
\be \la{rhobd}
\rho < \frac{1}{4\left (|b| +|c|\right )}.
\ee
and sufficiently small $\eps _2 $, (\ref{eq:star}),  contradicts with (\ref{12bd}).

\noindent
{\bf Case II: $\|[\theta_2, \theta_3]^\top\| >1/\sqrt{3} $ .} Follows similarly with the same set of $\rho$ given in (\ref{rhobd}).
\end{document}